\documentclass[12pt,a4paper]{article}
 \usepackage{graphicx}
\usepackage{amsfonts,amssymb,amsthm,amsmath}
\usepackage{newlfont}
\usepackage{enumitem}
\usepackage{color}
\usepackage{graphicx,color}
\usepackage{amsmath, amssymb, graphics}

\def\r{\mathbb R}

\newtheorem{theorem}{Theorem}[section]

\newtheorem{remark}[theorem]{Remark}
\newtheorem{proposition}[theorem]{Proposition}

\title {The Dirichlet problem for the $\alpha$-singular minimal surface  equation}

\author{Rafael L\'opez\footnote{Partially supported by the grant no. MTM2017-89677-P, MINECO/AEI/FEDER, UE.}\\
Departamento de Geometr\'{\i}a y Topolog\'{\i}a\\
 Instituto de Matem\'aticas (IEMath-GR)\\
 Universidad de Granada\\
 18071 Granada, Spain\\
\texttt{rcamino@ugr.es}
}

\date{}

\begin{document}
\maketitle
\begin{abstract}
Let $\Omega\subset\r^n$ be a bounded  mean convex domain. If $\alpha<0$, we prove the existence and uniqueness of classical solutions of the Dirichlet problem in $\Omega$ for the $\alpha$-singular minimal surface equation with arbitrary continuous boundary data. 
\end{abstract}
 


\section{Introduction and statement of results}
Let $\Omega\subset\r^n$ be a smooth   domain and $\alpha$ a given constant. We consider the   existence of classical  solutions $u\in C^2(\Omega)\cap C^0(\overline{\Omega})$, $u>0$ in $\overline{\Omega}$,  of the Dirichlet problem
\begin{eqnarray}
&&\mbox{div}\left(\dfrac{Du}{\sqrt{1+|Du|^2}}\right)= \frac{\alpha}{u\sqrt{1+|Du|^2}}\quad \mbox{in $\Omega$}\label{eq1}\\
&&u=\varphi\quad \mbox{on $\partial\Omega,$}\label{eq2}
\end{eqnarray}
where $D$ and div are the gradient and divergence operators and $\varphi>0$ is a positive continuous function in $\partial\Omega$. 
We call Equation   (\ref{eq1}) the {\it $\alpha$-singular minimal surface equation} and the graph $\Sigma_u=\{(x,u(x)):x\in\Omega\}$ is an {\it $\alpha$-singular minimal hypersurface}, or simply, a singular minimal surface.   Equation (\ref{eq1})  is an equation of mean curvature type because the mean curvature $H$ of $\Sigma_u$ is $H=\alpha/(nu\sqrt{1+|Du|^2})$. In the limit case  $\alpha=0$,  Equation (\ref{eq1}) is the known minimal surface equation. The theory of singular minimal surfaces has been intensively studied from the works of Bemelmans, Dierkes and Huisken among others: see \cite{bd,bht,di,di1,di2,dh,ke,ni}. An interesting case is   $\alpha=1$ because  the hypersurface $\Sigma_u$ has the property to have the lowest  center of gravity and this generalizes  to the $n$-dimensional case, the same property that  the catenary curve (\cite{bht,dh}). Other case of interest is $\alpha=-n$, where now $\Sigma_u$ is a minimal hypersurface in the upper halfspace model of hyperbolic space.

Usually, the existence of examples of singular minimal surfaces have been  considered from the parametric viewpoint  by solving the Plateau problem. However, the existence of singular minimal graphs has been only studied in \cite{bht} (see also \cite{di6}). Indeed, it was proved the existence of a solution of (\ref{eq1})-(\ref{eq2}) for $\alpha>0$ in bounded mean convex domains of $\r^n$ provided  the size of $\Omega$ is small in relation to the boundary data $\varphi$. Recall that $\Omega$ is said to be mean convex if the mean curvature $H_{\partial\Omega}$ of $\partial\Omega$ with respect to the inner normal is nonnegative at every point. Thus the result in \cite{bht} is an approach to the known   result of Jenkins and Serrin  in \cite{js}  that asserts the existence of a minimal graph for arbitrary continuous boundary data $\varphi$ if and only if $\Omega$ is a  bounded mean convex domain. 

The geometric properties of the singular minimal surfaces change drastically depending on the sign of $\alpha$. In this paper, and when $\alpha$ is negative, we are able  to extend the Jenkins-Serrin result without assumptions on the size of $\Omega$. The existence result is established by our next theorem. 

\begin{proposition} \label{t1}
Let $\Omega\subset\r^n$ be a bounded  mean convex domain with $C^{2,\gamma}$ boundary $\partial\Omega$ for some $\gamma\in (0,1)$. Assume $\alpha<0$. If $\varphi\in C^{2,\gamma}(\partial\Omega)$ is a positive function, then there exists a unique positive solution $u\in C^{2,\gamma}(\overline{\Omega})$ of (\ref{eq1})-(\ref{eq2}).
\end{proposition}

 If the assumption of the mean convexity of $\Omega$ fails at some point, we do not show that Theorem \ref{t1} is not longer true, that is, there exists a boundary data $\varphi$ for which no solution exists. The corresponding result for minimal graphs in hyperbolic space and constant boundary data was proved by Lin in \cite[Th. 2.1]{li}. The proof of Theorem \ref{t1} involves the continuity method by deforming (\ref{eq1})-(\ref{eq2}) in a uniparametric family of Dirichlet problems varying the value of $\alpha$, and the classical techniques of   apriori estimates for elliptic equations: we refer   the reader to \cite{gt} as a general reference. Our proof can not extend  to the case $\alpha>0$  by the absence of apriori $C^0$ estimates since if $\alpha>0$ we have to prevent that $|u|\rightarrow 0$  for a solution $u$ in the continuity method.  

This paper is organized as follows. In Section \ref{sec2} we recall the maximum and comparison principles for Equation (\ref{eq1}) as well as the behavior of the radial solutions. In Sections \ref{sec3} and \ref{sec4},  we deduce the height and gradient estimates, respectively, and finally,   the   last section \ref{sec5} presents the proof of the Theorem \ref{t1} following the known continuity method.

\section{Preliminaries}\label{sec2}

As a consequence of the maximum principle for elliptic equations of divergence type, we have:

  \begin{proposition}[Touching principle]\label{pr21} Let $\Sigma_i$  be two $\alpha$-singular minimal surfaces, $i=1,2$.   If   $\Sigma_1$ and  $\Sigma_2$ have a common tangent interior point   and $\Sigma_1$ lies above $\Sigma_2$ around $p$, then $\Sigma_1$ and $\Sigma_2$ coincide at an open set around $p$. 
\end{proposition}

We also need to state the known comparison principle in the context of $\alpha$-singular minimal surfaces. Define the operator 
\begin{equation}\label{op}
\begin{split}
Q[u] &= (1+|Du|^2)\Delta u-u_iu_ju_{ij}-\frac{\alpha(1+|Du|^2)}{u}\\
& = a_{ij}(Du)u_{ij}+{\mathbf b}(u,Du),
\end{split}
\end{equation}
 where 
 $$a_{ij}=(1+|Du|^2)\delta_{ij}-u_iu_j,\quad {\mathbf b}= - \frac{\alpha(1+|Du|^2)}{u}.$$
 Here we are denoting $u_i=\partial u/\partial x_i$, $1\leq i\leq n$,  and we assume the summation convention of repeated indices.    It is immediate that $u$ is a solution of Equation (\ref{eq1}) if and only if $Q[u]=0$.  Further observe that the function $\mathbf{b}$ is non-increasing in $u$ for each $(x,Du)\in\Omega\times\r^n$ because $\alpha<0$. In particular, we   apply the classical comparison principle for elliptic equations (\cite[Th. 10.1]{gt}).

\begin{proposition}[Comparison principle] Let $\Omega\subset\r^n$ be a bounded domain. If $u,v\in C^2(\Omega)\cap C^0(\overline{\Omega})$ satisfy $Q[u]\geq Q[v]$ and $u\leq v$ on $\partial\Omega$, then $u\leq v$ in $\Omega$.
\end{proposition}
 
 We now prove the uniqueness of solutions of (\ref{eq1})-(\ref{eq2}) when $\alpha$ is negative.
 
\begin{proposition}\label{pr-u}
Let $\Omega\subset\r^n$ be a bounded domain and $\alpha<0$.  The solution of (\ref{eq1})-(\ref{eq2}), if exists, is unique.
  \end{proposition}   
  \begin{proof}
   The uniqueness is a consequence that the right hand side of (\ref{eq1}) is non-decreasing on $u$ (\cite[Th. 10.1]{gt}).
  \end{proof}
  
We point out that the above result fails if $\alpha>0$ by taking suitable examples in the class of rotational $\alpha$-singular minimal surfaces. 
  
We now show the behavior of the radial solutions of Equation (\ref{eq1}). Denote $u=u(r)$, $r=|x|$, and subsequently, $\Sigma_u$ is a rotational singular minimal hypersurface. The behavior of $u$ depends strongly on the sign of $\alpha$: see \cite{ke,lo}.  For our purposes, we only need the case   $\alpha<0$.

\begin{proposition}\label{pr-rot} Let $\alpha<0$ and let $u=u(r)$ be a radial solution of (\ref{eq1}). Then $u$ is a concave function whose maximal domain is a bounded ball $B_R=\{x\in\r^n: |x|<R\}$ with 
$$\lim_{r\rightarrow R} u(r)=0,\quad \lim_{r\rightarrow  R} u'(r)=-\infty.$$
\end{proposition}

 Recall that homotheties from the origin ${\mathbf O}\in\r^n$ preserve the Equation (\ref{eq1}), that is, if $u$ is a solution of (\ref{eq1}), then $\lambda u(x/\lambda)$, $\lambda>0$, also satisfies (\ref{eq1}).  As a consequence of Proposition \ref{pr-rot} and  using  homotheties, we establish the solvability of (\ref{eq1})-(\ref{eq2}) when $\Omega$ is any arbitrary ball and $\varphi$ is any positive constant. 
 
 \begin{proposition}\label{pr25} Let $\alpha<0$. Then for any $r, c>0$, there exists a unique radial  solution $u$ of  (\ref{eq1}) in $\Omega=B_r$ with  $u=c$ on $\partial B_r$.
 \end{proposition}

\begin{proof} Let $u=u(r)$ be any radial solution of (\ref{eq1}) defined in its maximal domain $B_R$. Take $\lambda>0$ sufficiently large so $u_\lambda(x)=\lambda u(x/\lambda)$ has the property that the bounded domain determined by its graph $\Sigma_{u_\lambda}$ and the plane $\r^n\times\{0\}$ contains the ball $B_r\times\{c\}$. Let $\lambda$ decrease until some value  $\lambda_0$ such that $\Sigma_{u_{\lambda_0}}$ intersect $B_r\times\{c\}$. Then the function $u_{\lambda_0}$ is the solution that we are looking for.
\end{proof}

 \section{Height estimates}\label{sec3}
 
 In this section we obtain $C^0$ apriori estimates for solutions of (\ref{eq1})-(\ref{eq2}) when $\alpha<0$.
 
  \begin{proposition}  \label{pr-31}
   Let $\Omega\subset\r^n$ be a bounded domain and   $\alpha<0$. If  $u$ is a positive solution of (\ref{eq1})-(\ref{eq2}), then there exists a constant $C_1=C_1(\alpha,\Omega,\varphi)>0$  such that 
\begin{equation}\label{eh}
 \min_{\partial\Omega}\varphi\leq u\leq C_1 \quad \mbox{in $\Omega$}.
\end{equation} 
\end{proposition} 
  
  \begin{proof} Since   the right hand side of (\ref{eq1}) is  negative, then $\inf_\Omega u=\min_{\partial\Omega}\varphi$ by the maximum principle. The upper estimate for $u$ is obtained by comparing $\Sigma_u$ with radial solutions of (\ref{eq1}). Exactly, let  $B_R\subset\r^n$ be a  ball centered at the origin ${\mathbf O}$ of radius $R>0$ sufficiently large such that $\overline{\Omega}\subset B_R$.      Set $\varphi_M=\max_{\partial\Omega}\varphi$. By Proposition \ref{pr25},  let  $v=v(r)$ be the radial solution of (\ref{eq1}) with $v=\varphi_M$ on $\partial B_R$. 
    
 Let $\lambda>1$ be sufficiently large that $\lambda\Sigma_v\cap\Sigma_u=\emptyset$. Notice that the hypersurface $\lambda\Sigma_v$ is a singular minimal hypersurface for the same constant $\alpha$ than $\Sigma_v$. Let    $\lambda$ decrease to $1$. By Proposition \ref{pr21}, it is not possible a contact at some interior point between $\lambda \Sigma_v$ and $\Sigma_u$  because  $\partial(\lambda\Sigma_v)\cap\partial\Sigma_u=\lambda(\partial\Sigma_v)\cap\partial\Sigma_u=\emptyset$ for all $\lambda>1$. Therefore we arrive until  the initial position $\lambda=1$ and we find $\Sigma_v\cap\Sigma_u=\emptyset$. Consequently, $u<v\leq \sup_\Omega v:=C_1$, and $C_1$ depends only on $\alpha$, $\Omega$ and $\varphi$. 
   \end{proof}
 
Of particular interest is when  $\varphi=c>0$ is a constant function on $\partial\Omega$. Then we may improve estimate (\ref{eh}) with the preceding argument by taking all singular minimal surfaces of rotational type. Among all them, we choose the rotational example $\Sigma_v$ with lowest height. This achieves when, after a horizontal translation if necessary, consider $B_R$   the circumscribed sphere of $\Omega$. In such a case, the inequality (\ref{eh}) is now $c<u\leq v\leq v(0)$ in $\Omega$.
 
 \section{Gradient estimates}\label{sec4}

Firstly, we derive estimates for $\sup_\Omega|Du|$ in terms of $\sup_{\partial\Omega}|Du|$. In the next result, the fact that $\alpha$ is negative is essential.

\begin{proposition}[Interior gradient estimates] \label{pr-41} 
   Let $\Omega\subset\r^n$ be a bounded domain and $\alpha<0$. If  $u\in C^2(\Omega)\cap C^1(\overline{\Omega})$ is a positive solution of (\ref{eq1})-(\ref{eq2}), then the maximum of the gradient is attained at some boundary point, that is, 
  $$\max_{\overline{\Omega}}|Du|=\max_{\partial\Omega}|Du|.$$
\end{proposition} 

\begin{proof}
 We know that (\ref{eq1}) can be expressed as  (\ref{op}).  Let $v^k=u_k$, $1\leq k\leq n$, and we differentiate (\ref{op}) with respect to $x_k$, obtaining for each $k$,
 \begin{equation}\label{eq3}
 \left((1+|Du|^2)\delta_{ij}-u_iu_j\right)v_{ij}^k+2\left(u_i\Delta u-u_ju_{ij}-\frac{\alpha u_i}{u}\right)v_i^k+\frac{\alpha(1+|Du|^2)}{u^2}v^k=0.
 \end{equation}
 Equation (\ref{eq3})  is a linear elliptic equation in the function $v^k$ and, in addition,  the coefficient for $v^k$ is negative because $\alpha<0$. By the maximum principle \cite[Th. 3.7]{gt},   $|v^k|$, and then $|Du|$, has not an interior maximum. In particular, if $u$ is a solution of (\ref{eq1}), the maximum of $|Du|$ on the compact set $\overline{\Omega}$ is attained at some boundary point,  proving the result.  
 \end{proof}

Once proved   Proposition \ref{pr-41},  the problem of  finding apriori estimates of $|Du|$ reduces to find them along $\partial\Omega$. Then we now address it by proving  that $u$  admits   barriers from above and from below along $\partial\Omega$. It is now when we use the mean convexity property of $\Omega$.

     \begin{proposition}[Boundary gradient estimates]  \label{pr42}
   Let $\Omega\subset\r^n$ be a  bounded mean convex domain and $\alpha<0$. If  $u\in C^2(\Omega)\cap C^1(\overline{\Omega})$ is a positive solution of (\ref{eq1})-(\ref{eq2}), then there exists a constant $C_2=C_2(\alpha,\Omega, C_1,\|\varphi\|_{2;\Omega})$ such that 
   $$\max_{\partial\Omega}|Du|\leq C_2.$$
    
\end{proposition}

 \begin{proof} 
 We consider the operator $Q[u]$ defined (\ref{op}). A lower barrier for $u$  is obtained by considering the subsolution $v^0$ of the Dirichlet problem for the minimal surface equation in $\Omega$ with the same boundary data $\varphi$: the existence of $v^0$ is assured by the   Jenkins-Serrin  result (\cite{js}). Because $Q[v^0]>0=Q[u]$ and $v^0=u$ on $\partial\Omega$, we conclude  $v^0<u$ in $\Omega$ by the comparison principle.

  We now find an upper barrier for $u$. Here we use  the distance function in a small tubular neighborhood of $\partial\Omega$ in $\Omega$.   The following arguments are standard: see \cite[Ch. 14]{gt} for details. 
  Consider  the distance function $d(x)=\mbox{dist}(x,\partial\Omega)$ and let $\epsilon>0$ sufficiently small so  $\mathcal{N}_\epsilon=\{x\in\overline{\Omega}: d(x)<\epsilon\}$ is a tubular neighborhood of $\partial\Omega$. The value of $\epsilon$ will be precised later. We can parametrize $\mathcal{N}_\epsilon$  using normal coordinates $x\equiv (t,\pi(x)) \in\mathcal{N}_\epsilon$, where we write $x=\pi(x)+t\nu(\pi(x))$ for some $t\in [0,\epsilon)$, where $\pi:\mathcal{N}_\epsilon\rightarrow\partial\Omega$ is the orthogonal projection and $\nu$ is the unit   normal vector to $\partial\Omega$ pointing to $\Omega$. Among the properties of the function $d$, we know that $d$ is   $C^2$,  $|Dd|(x)=1$,  and $\Delta d(x) \leq -(n-1)H_{\partial\Omega}(\pi(x))\leq 0$  for all $x\in\mathcal{N}_\epsilon$, where    last inequality holds because $\Omega$ is mean convex.  

Define in $\mathcal{N}_\epsilon$ a function $w=h\circ d+\varphi$, where  we extended $\varphi$ to $\mathcal{N}_\epsilon$ by letting $\varphi(x)=\varphi(\pi(x))$. Here $h(t)=a\log(1+bt)$, $a,b>0$ to be chosen later. It is known that $h\in C^\infty[0,\infty)$ and $h''=-h'^2/a$. The computation of  $Q[w]$ leads to
$$Q[w]=a_{ij}(h''d_id_j+h'd_{ij}+\varphi_{ij})-\frac{\alpha}{w}(1+|Dw|^2).$$
From $|Dd|=1$, it follows that $\langle D(Dd)_x\xi,Dd(x)\rangle=0$ for all $\xi\in\r^n$. If  $\{e_i\}_i$ is the canonical basis of $\r^n$, by taking $\xi=e_i$, we find $d_{ij}d_j=0$. Thus
\begin{eqnarray*}
w_iw_jd_{ij}&=&(h'd_i+\varphi_i)(h'd_j+\varphi_j)d_{ij}=(h'^2d_i+2h'\varphi_i)d_jd_{ij}+\varphi_i \varphi_jd_{ij}\\
&=&\varphi_{i}\varphi_jd_{ij}\geq |D\varphi|^2\Delta d.
\end{eqnarray*}
Using this inequality and from the definition of $a_{ij}$ in (\ref{op}), we derive
$$a_{ij}d_{ij}=(1+|Dw|^2)\Delta d-w_iw_j d_{ij}\leq(1+|Dw|^2-|D\varphi|^2)\Delta d.$$
 Since $|\xi|^2\leq a_{ij}\xi_i\xi_j\leq(1+|Dw|^2)|\xi|^2$ for all $\xi\in\r^n$,   we have
 $a_{ij}d_id_j\geq 1$ and $a_{ij}\varphi_{ij}\leq (1+|Dw|^2)|D^2\varphi|$, where $|D^2\varphi|=\sum_{ij}\sup_{\overline{\Omega}}|\varphi_{ij}|$. By using that $h'>0$ and $\Delta d\leq 0$, we find
\begin{eqnarray*}
Q[w]&\leq& h''+h'\Delta d(1+|Dw|^2-|D\varphi|^2)+(-\frac{\alpha}{w}+|D^2 \varphi|)(1+|Dw|^2)\\
&\leq & h''+\left(-\frac{\alpha}{w}+|D^2 \varphi|\right)(1+|Dw|^2)\\
&=&h''+\left(-\frac{\alpha}{w}+|D^2 \varphi|\right)(1+h'^2+|D\varphi|^2+2h'|D\varphi|).
\end{eqnarray*}
In the tubular neighborhood $\mathcal{N}_\epsilon$, we have
\begin{equation}\label{ww}
w=a\log(1+bd)+\varphi\geq a\log(1+b)-\|\varphi\|_{0;\Omega}>0,
\end{equation}
where the last inequality holds  if $a\log(1+b)$ is sufficiently large. We will now assume that this is true. In particular, and because $\alpha<0$, we find 
$$-\frac{\alpha}{w}+|D^2 \varphi|\leq \frac{-\alpha}{a\log(1+b)-\|\varphi\|_{0;\Omega}}+\|D^2\varphi\|_{0;\Omega}:=\beta.$$
Therefore, and  taking into account that $h''=-h'^2/a$, we deduce
$$Q[w]\leq \left(\beta-\frac{1}{a}\right)h'^2+2\beta h'\|D\varphi\|_{0;\Omega}+\beta(1+\|D\varphi\|_{0;\Omega}^2).$$
We take $a= c/\log(1+b)$, $c>0$ to be chosen later. Then the above inequality for $Q[w]$ writes as 
\begin{equation}\label{qq}
Q[w]\leq \left(\beta-\frac{\log(1+b)}{c}\right)\frac{c^2b^2}{(1+bt)^2}+2\beta  \|D\varphi\|_{0;\Omega} \frac{cb}{1+bt}+\beta(1+\|D\varphi\|_{0;\Omega}^2),
\end{equation}
where we denote again $x\equiv(t,\pi(x))$ in normal coordinates. For $b$ sufficiently large, the parenthesis $\beta-\log(1+b)/c$ in (\ref{qq}) is negative. If we see  the right hand side in (\ref{qq}) as a continuous function  $\phi(t)$, $t>0$,   then we find that $\phi(0)<0$ for $b$ large enough. Since $\partial\Omega$ is compact, by an argument of continuity, there exists $\epsilon>0$ small enough to ensure that $\phi(t)<0$ for $t\in [0,\epsilon)$. This defines definitively the   tubular neighborhood $\mathcal{N}_\epsilon$ of $\partial\Omega$ and, furthermore, we conclude that   for $b$ large enough, we find $Q[w]<0$.     
   
 In order to assure that $w$ is a local upper barrier  in $\mathcal{N}_\epsilon$ for the Dirichlet problem (\ref{eq1})-(\ref{eq2}), and because we will  apply the comparison principle, we have to prove that 
\begin{equation}\label{mm}
u\leq w\quad \mbox{in $\partial\mathcal{N}_\epsilon$}.
\end{equation}
   In $\partial\mathcal{N}_\epsilon\cap\partial\Omega$,   the distance function  is $d=0$, so $w=\varphi=u$. On the other hand, in $\partial\mathcal{N}_\epsilon\setminus\partial\Omega$, we find $w=h(\epsilon)+\varphi=a\log(1+b\epsilon)+\varphi$. Denote $\mu=C_1+\|\varphi\|_{0;\Omega}$, where   $C_1$ is the constant of Proposition \ref{pr-31}.  Take $c>0$ sufficiently large so that
   $$c\geq\frac{\mu\log(1+b)}{\log(1+b\epsilon)}.$$
 With this choice of $c$, we infer that  $u\leq w$ in $\partial\mathcal{N}_\epsilon\setminus\partial\Omega$. By the way, and taking $c$ large enough if necessary, we assure that $w>0$ in (\ref{ww}). Definitively, (\ref{mm}) holds  in $\partial\mathcal{N}_\epsilon\setminus\partial\Omega$. Because $Q[w]<0=Q[u]$, we conclude     $u\leq w$ in $\mathcal{N}_\epsilon$ by the comparison principle. 
  
Consequently, we have proved the existence of lower and upper barriers for $u$  in $\mathcal{N}_\epsilon$, namely, $v^0\leq u\leq w$ in $\mathcal{N}_\epsilon$. Hence we deduce 
$$\max_{\partial\Omega}|Du|\leq C_2:=\max\{\|Dw\|_{0;\partial\Omega}, \|Dv^0\|_{0;\partial\Omega}\}$$
  and both values $\|Dw\|_{0;\partial\Omega}, \|Dv^0\|_{0;\partial\Omega}$ depend only on $\alpha$, $\Omega$, $C_1$ and $\varphi$.  This completes the proof of proposition. 
    \end{proof}

 \section{Proof of Theorem \ref{t1}}\label{sec5}

We establish the solvability  of the Dirichlet problem   (\ref{eq1})-(\ref{eq2}) by applying  a slightly modified method of continuity,   where the boundary data is fixed on the deformation (see   \cite[Sec. 17.2]{gt}). Define  the family of Dirichlet  problems parametrized by $t\in [0,1]$ by  
 $$\mathcal{P}_t: \left\{\begin{array}{cll}
Q_t[u]&=&0 \mbox{ in $\Omega$}\\
 u&=& \varphi \mbox{ on $\partial\Omega,$}
 \end{array}\right.$$
 where 
   $$Q_t[u]= (1+|Du|^2)\Delta u-u_iu_ju_{ij}-\frac{\alpha t(1+|Du|^2)}{u}.$$
The graph $\Sigma_{u_t}$ of a solution of $u_t$   is a $(t\alpha)$-singular minimal surface.    As usual, let 
$$\mathcal{A}=\{t\in [0,1]: \exists u_t\in C^{2,\gamma}(\overline{\Omega}), u_t>0, Q_t[u_t]=0, {u_t}_{|\partial\Omega}=\varphi\}.$$ 
The proof consists to show that $1\in \mathcal{A}$. For this, we prove that $\mathcal{A}$ is a non-empty open and closed subset of $[0,1]$.

\begin{enumerate}
\item  The set  $\mathcal{A}$ is not empty. Let us observe that $0\in\mathcal{A}$: if $t=0$, then $u_0$ is   the solution $v^0$ provided by the Jenkins-Serrin theorem (\cite{js}).  Notice that $v^0>0$ by the maximum principle.

\item The set $\mathcal{A}$ is open in $[0,1]$. Given $t_0\in\mathcal{A}$ we need to prove that there exists $\epsilon>0$ such that $(t_0-\epsilon,t_0+\epsilon)\cap [0,1]\subset\mathcal{A}$. Define the map $T(t,u)=Q_t[u]$ for $t\in\r$ and $u\in  C^{2,\gamma}(\overline{\Omega})$. Then $t_0\in\mathcal{A}$ if and only if $T(t_0,u_{t_0})=0$. If we prove that the derivative  of $Q_t$ with respect to $u$, say $(DQ_t)_u$, at the point $u_{t_0}$ is an isomorphism, it follows from the Implicit Function Theorem the existence of an open set $\mathcal{V}\subset C^{2,\gamma}(\overline{\Omega})$, with $u_{t_0}\in \mathcal{V}$ and a $C^1$ function $\xi:(t_0-\epsilon,t_0+\epsilon)\rightarrow \mathcal{V}$ for some $\epsilon>0$, such that $\xi(t_0)=u_{t_0}>0$ and  $T(t,\xi(t))=0$ for all $t\in (t_0-\epsilon,t_0+\epsilon)$: this guarantees that $\mathcal{A}$ is an open  set of  $[0,1]$.

The proof that $(DQ_t)_u$ is one-to-one is equivalent that say that for any $f\in C^\gamma(\overline{\Omega})$, there exists a unique solution $v\in C^{2,\gamma}(\overline{\Omega})$ of the linear equation $Lv:=(DQ_t)_u(v)=f$ in $\Omega$ and $v=\varphi$ on $\partial\Omega$. The computation of $L$ is
$$Lv=(DQ_t)_uv=a_{ij}(Du)v_{ij}+\mathcal{B}_i(u,Du,D^2u)v_i+{\mathbf c}(u,Du)v,$$
where $a_{ij}$ is as in (\ref{op}) and 
$$\mathcal{B}_i=2( \Delta u-\frac{\alpha t}{u})u_i-2u_ju_{ij},\quad {\mathbf c}=\frac{\alpha t(1+|Du|^2)}{u^2}.$$
Since $\alpha<0$, the function $\textbf{c}$ satisfies ${\mathbf c}\leq 0$ and the existence and uniqueness is assured by standard theory (\cite[Th. 6.14]{gt}).

\item The set $\mathcal{A}$ is closed in $[0,1]$. Let $\{t_k\}\subset\mathcal{A}$ with $t_k\rightarrow t\in [0,1]$. For each $k\in\mathbb{N}$, there exists $u_k\in C^{2,\gamma}(\overline{\Omega})$, $u_k>0$,  such that $Q_{t_k}[u_k]=0$ in $\Omega$ and $u_k=\varphi$ in $\partial\Omega$. Define the set
$$\mathcal{S}=\{u\in C^{2,\gamma}(\overline{\Omega}): \exists t\in [0,1]\mbox{ such that }Q_{t}[u]=0 \mbox{ in }\Omega, u_{|\partial\Omega}=\varphi\}.$$
Then $\{u_k\}\subset\mathcal{S}$. If we prove that the set $\mathcal{S}$ is bounded in $C^{1,\beta}(\overline{\Omega})$ for some $\beta\in[0,\gamma]$, and since $a_{ij}=a_{ij}(Du)$ in (\ref{op}), then Schauder theory proves that $\mathcal{S}$ is bounded in $C^{2,\beta}(\overline{\Omega})$, in particular, $\mathcal{S}$ is precompact in $C^2(\overline{\Omega})$  (see Th. 6.6 and Lem. 6.36 in \cite{gt}). Thus there exists a subsequence $\{u_{k_l}\}\subset\{u_k\}$ converging to some $u\in C^2(\overline{\Omega})$ in $C^2(\overline{\Omega})$. Since $T:[0,1]\times C^2(\overline{\Omega})\rightarrow C^0(\overline{\Omega})$ is continuous, it follows $Q_t[u]=T(t,u)=\lim_{l\rightarrow\infty}T(t_{k_l},u_{k_l})=0$ in $\Omega$. Moreover, $u_{|\partial\Omega}=\lim_{l\rightarrow\infty} {u_{k_l}}_{|\partial\Omega}=\varphi$ on $\partial\Omega$, so $u\in C^{2,\gamma}(\overline{\Omega})$ and consequently, $t\in \mathcal{A}$.

The above reasoning says that   $\mathcal{A}$ is  closed in $[0,1]$ provided we find a constant $M$ independent of $t\in\mathcal{A}$,  such that  
$$
\|u_t\|_{C^1(\overline{\Omega})}=\sup_\Omega |u_t|+\sup_\Omega|Du_t|\leq M.
$$
However the $C^0$ and $C^1$ estimates for $u_1=u_t$, that is, when the parameter   is $t=1$,  proved in Sections \ref{sec3} and \ref{sec4} are enough as we     see now.  

The $C^0$ estimates for $u_t$ follow with the comparison principle. Indeed, let $t_1<t_2$, $t_i\in [0,1]$, $i=1,2$. Then $Q_{t_1}[u_{t_1}]=0$ and 
$$Q_{t_1}[u_{t_2}]=-\frac{(t_1-t_2)\alpha(1+|Du_{t_2}|^2)}{u_{t_2}}<0$$
because $\alpha<0$. Since $u_{t_1}=u_{t_2}$ on $\partial\Omega$, the comparison principle yields $u_{t_1}<u_{t_2}$ in $\Omega$. This proves that the solutions $u_{t_i}$ are ordered in increasing sense according the parameter $t$. Consequently, and by (\ref{eh}), we find
\begin{equation}\label{ut}
\sup_\Omega u_t \leq \sup_\Omega u_1 \leq C_1.
\end{equation}

 In order to find  the gradient estimates for the solution $u_t$,  the same computations  given in   Proposition \ref{pr42}     conclude that $\sup_{\partial\Omega}|Du_t|$ is bounded by a constant depending on $\alpha$, $\Omega$, $\varphi$ and  $\|u_t\|_{0;\Omega}$. However, and by using (\ref{ut}),  the value  $\|u_t\|_{0;\Omega}$ is bounded by $C_1$, which depends only on $\alpha$, $\varphi$ and $\Omega$.   

\end{enumerate}
The above three steps proves the part of existence in Theorem \ref{t1}. The uniqueness is consequence  of Proposition \ref{pr-u} and this completes the proof of theorem.

\begin{remark} We point out that a $C^0$-version of Theorem \ref{t1} holds for continuous positive boundary values $\varphi$. For this, let $\varphi\in C^0(\partial\Omega)$ be given. Let $\{\varphi_k^{+}\}, \{\varphi_k^{-}\}\in C^{2,\gamma}(\partial\Omega)$ be a monotonic sequence of functions converging from above and from below to $\varphi$ in the $C^0$ norm. It follows from Theorem \ref{t1}  that there exist solutions $u_k^{+}, u_k^{-}\in C^{2,\gamma}(\overline{\Omega})$ of the $\alpha$-singular minimal surface equation (\ref{eq1}) such that ${u_k^{+}}_{|\partial\Omega}=\varphi_k^{+}$ and ${u_k^{-}}_{|\partial\Omega}=\varphi_k^{-}$. The sequences $\{u_k^{\pm}\}$ are uniformly bounded in the $C^0$ norm since, by the comparison principle, we find
$$u_1^{-}\leq\ldots \leq u_k^{-}\leq u_{k+1}^{-}\leq\ldots \leq u_{k+1}^{+}\leq u_k^+\leq\ldots\leq u_1^{+}$$
for every $k$. By the proof of Theorem \ref{t1}, the sequences $\{u_k^{\pm}\}$ have a priori $C^1$ estimates depending only on $\alpha$, $\Omega$, $\varphi$ and  the $C^0$ estimates.    Using classical Schauder theory again (\cite[Th. 6.6]{gt}),  the sequence $\{u_k^{\pm}\}$ contains a subsequence  $\{v_k\}\in C^{2,\gamma}(\overline{\Omega})$ converging uniformly on the $C^2$ norm on compacts subsets of $\Omega$ to a solution $u\in C^2(\Omega)$ of (\ref{eq1}). Since $\{{u_k^{\pm}}_{|\partial\Omega}\}=\{\varphi_k^{\pm}\}$    and $\{\varphi_k^{\pm}\}$  converge to $\varphi$, it follows  that $u$ extends continuously to $\overline{\Omega}$ and $u_{|\partial\Omega}=\varphi$.

\end{remark}




\begin{thebibliography}{99}



\bibitem{bd} J. Bemelmans, U. Dierkes, \textit{On a singular variational integral with linea growth, I: Existence and regularity of minimizers}. Arch. Rat. Mech. Analysis, 100 (1987), 83--103.
 
\bibitem{bht} R. B\"{o}hme, S. Hildebrandt, E. Taush,
\textit{The two-dimensional analogue of the catenary}.
 Pacific J. Math. 88 (1980), 247--278.
 



 

\bibitem{di} U. Dierkes,
\textit{A geometric maximum principle, Plateau's problem for surfaces of
prescribed mean curvature, and the two-dimensional analogue of the catenary}, 
in: Partial Differential Equations and Calculus of Variations, pp. 116-141,
Springer Lecture Notes in Mathematics 1357 (1988).

\bibitem{di1} U. Dierkes,  \textit{A Bernstein result for energy minimizing hypersurfaces}, Calc. Var. Partial Differential Equations \textbf{1} (1993), no. 1, 37--54.

\bibitem{di6} U. Dierkes,  \textit{On the regularity of solutions for a singular variational problem}, Math. Z. \textbf{225} (1997),   657--670.

\bibitem{di2} U. Dierkes,
Singular minimal surfaces. \textit{Geometric analysis and nonlinear partial differential equations}, 177--193, Springer, Berlin, 2003.

\bibitem{dh} U. Dierkes,  G. Huisken,
\textit{The $n$-dimensional analogue of the catenary: existence and nonexistence}.
Pacific J. Math. \textbf{141} (1990), 47--54.




\bibitem{gt} D. Gilbarg, N. S. Trudinger,   \textit{Elliptic Partial Differential Equations of Second Order}. Second edition.   Springer-Verlag, Berlin, 1983.


\bibitem{js} H. Jenkins, J. Serrin, \textit{The Dirichlet problem for the minimal surface equation in higher dimensions}, J. Reine Angew. Math. \textbf{229} (1968), 170--187.
 

\bibitem{ke} J. B. Keiper, \textit{The axially symmetric $n$-tectum}, preprint, Toledo University (1980).
 
 \bibitem{li} F. H. Lin, \textit{On the Dirichlet problem for minimal graphs in hyperbolic space}. Invent. Math. \textbf{96} (1989), 593--612.
 
\bibitem{lo} R. L\'opez, \textit{Invariant singular minimal surfaces},  Ann. Global Anal. Geom. \textbf{53}   (2018), 521--541. 
  
  \bibitem{ni} J. C. C. Nitsche,
\textit{A non-existence theorem for the two-dimensional analogue of the catenary}. 
Analysis, \textbf{6}  (1986), 143--156.
  

 


\bibitem{se} J. B. Serrin, \textit{The problem of Dirichlet for quasilinear elliptic differential equations with many independent variables}. Phil. Trans. R. Soc. Lond. \textbf{264} (1969), 413--496.

 
 \end{thebibliography}
 \end{document}